\documentclass[10pt, a4paper, reqno]{amsart}
\usepackage{amscd}
\usepackage{amsmath}
\usepackage{amstext}
\usepackage{amsthm}
\usepackage{amssymb}
\usepackage{amsopn}
\usepackage{amssymb}
\usepackage{amsxtra}
\usepackage{amsfonts}
\usepackage{fancyhdr}
\usepackage{paralist, epsfig}
\usepackage[mathcal]{euscript}
\usepackage[english]{babel}
\usepackage[latin1]{inputenc}
\usepackage{diagrams}

\usepackage[pdftex]{color}

\newtheoremstyle{normal}
{5pt}
{5pt}
{\normalfont}
{}
{\bfseries}
{}
{0.4em}
{\bfseries{\thmname{#1}\thmnumber{ #2}.\thmnote{ \hspace{0.5em}(#3)\newline}}}

\newtheoremstyle{kursiv}
{5pt}
{5pt}
{\itshape}
{}
{\bfseries}
{}
{0.4em}
{\bfseries{\thmname{#1}\thmnumber{ #2}.\thmnote{ \hspace{0.5em}(#3)\newline}}}

\theoremstyle{kursiv}
\newtheorem{theorem}{Theorem}[section]
\newtheorem{lemma}[theorem]{Lemma}
\newtheorem{proposition}[theorem]{Proposition}

\theoremstyle{normal}
\newtheorem{definition}[theorem]{Definition}
\newtheorem{example}[theorem]{Example}
\newtheorem{rem}[theorem]{Remark}

\newtheorem*{acknowledgement}{Acknowledgement}

\newcommand{\cok}{\operatorname{cok}\nolimits}
\renewcommand{\epsilon}{\varepsilon}
\newcommand{\columnvec}[2]{\genfrac{[}{]}{0pt}{}{\,#1\,}{#2}}
\newcommand{\rowvec}[2]{[\,#1\;#2\,]}
\newcommand{\id}{\operatorname{id}\nolimits}
\newcommand{\ind}[1]{\ensuremath{\operatorname{ind}_{#1}\nolimits}}

\begin{document}

\title{Maximal exact structures on additive categories}

\author{~}

\renewcommand{\thefootnote}{}
\hspace{-1000pt}\footnote{\hspace{5.5pt}2010 \emph{Mathematics Subject Classification}: Primary 46M18; Secondary 18E10, 18G50.}
\hspace{-1000pt}\footnote{\hspace{5.5pt}\emph{Key words}: exact category, quasi-abelian, semi-abelian, homological methods in functional analysis.\vspace{1.6pt}}
\hspace{-1000pt}\footnote{$^{\text{a}}$\,Faculty IV - Mathematics, Building E, University of Trier, 54286 Trier, Germany, Phone: +49\hspace{1.2pt}651\linebreak\phantom{x}\hspace{12.5pt}201\hspace{1.2pt}349\hspace{1.2pt}5, Fax: +49\hspace{1.2pt}651\hspace{1.2pt}201\hspace{1.2pt}347\hspace{1.2pt}9, e-Mail: dsieg80@googlemail.com.\vspace{1.6pt}}
\hspace{-1000pt}\footnote{\hspace{-6.8pt}$^{\text{b},1}$\,Corresponding author: FB C -- Mathematics, University of Wuppertal, Gau\ss{}str.~20, 42119 Wuppertal,\linebreak\phantom{x}\hspace{12.5pt}Germany, Phone: +49\hspace{1.2pt}202\hspace{1.2pt}439\hspace{1.2pt}253\hspace{1.2pt}1, Fax: +49\hspace{1.2pt}202\hspace{1.2pt}439\hspace{1.2pt}372\hspace{1.2pt}4, eMail: wegner@math.uni-wuppertal.de.}


\begin{abstract}\vspace{29pt}
We show that every additive category with kernels and cokernels admits a maximal exact structure. Moreover, we discuss two examples of categories of the latter type arising from functional analysis.
\end{abstract}

\maketitle

\begin{picture}(0,0)
\put(114,120){{\sc\MakeLowercase{Dennis Sieg}}\,$^{\text{a}}$ {\sc and} {\sc\MakeLowercase{Sven-Ake Wegner}}\,$^{\text{b}, 1}$}
\put(171,101){\small May 22, 2011}
\put(0,82){\small\it Dedicated to our teachers Klaus D. Bierstedt and Susanne Dierolf, who both died much too early.}

\end{picture}

\vspace{-30pt}

\section{Introduction}

The notion of an exact category in the context of additive categories was introduced by Quillen \cite{Quillen1973}. A comprehensive and self-contained exposition of the theory has recently been laid down by B\"uhler \cite{Buehler2009}. Quillen's definition is extrinsic, that is given a category $\mathcal{A}$, one has to specify a class $\mathcal{E}$ of composable morphisms, called short exact sequences, which satisfies certain axioms; in this case $\mathcal{E}$ is called an exact structure on $\mathcal{A}$. In the case of an additive category, there is always a smallest exact structure, namely the class formed by all short sequences which are split exact. On the other hand it is not clear if a given additive category admits a largest exact structure. However, if $\mathcal{A}$ is additive and has kernels and cokernels it frequently happens that the class of all kernel-cokernel pairs forms an exact structure which then by definition is the largest one. Categories of this type are called quasi-abelian and have 
been studied independently and under different names by several authors, see the historical remarks in Rump \cite[section 2]{Rump2008}; a detailed exposition can be found in the book of Schneiders \cite{Schneiders1999}. Quasi-abelian categories appear in different branches of mathematics. In particular they are the starting point for the use of homological methods in functional analysis since the categories of topological vector spaces and locally convex spaces as well as several subcategories, like Fr\'{e}chet or Banach spaces, are all quasi-abelian (but not abelian). The use of homological methods in functional analysis was started by Palamodov \cite{Palamodov1968, Palamodov1971}, re-invented by Vogt \cite{VogtLectures} and expanded by many others (see the book of Wengenroth \cite{Wengenroth} for detailed references and concrete applications). In parallel Prosmans \cite{Prosmans2000} studied the derived category of several quasi-abelian categories arising in functional analysis.
\smallskip
\\It is well-known that there are additive categories with kernels and cokernels which are not quasi-abelian, e.g. the category of complete topological vector spaces, cf.~\cite[section 3]{Prosmans2000}. Another notion in this context is that of being semi-abelian, see again \cite{Rump2008} for historical information. In the latter article Rump constructed a counter example to a conjecture of Raikov, which stated that each semi-abelian category is quasi-abelian. An earlier counterexample is contained in the article \cite{BonetDierolf2005} of Bonet, Dierolf and will be discussed in section \ref{Applications}. The failure of Raikov's conjecture implies in particular that a semi-abelian category in general is not exact in an intrinsic way.
\smallskip
\\Nonetheless, in this article we show that every additive category with kernels and cokernels admits a canonical largest exact structure, which can be described in terms of stability properties of morphisms under pullbacks and pushouts. This exact structure coincides with the class of all kernel-cokernel pairs if and only if the category is quasi-abelian. Thus, every additive category with kernels and cokernels (in particular, every semi-abelian category) admits an extrinsic but natural exact structure. 
\smallskip
\\Concerning applications in functional analysis, the above turns out to be very useful since there are many categories which are not quasi-abelian but additive with kernels and cokernels and thus accessible for homological algebra when endowed with an exact structure. In addition to the two examples which we will discuss in section \ref{Applications}, certain categories defined by projective and inductive limits which arise naturally in functional analysis, like that of PLB-, PLS- and PLN-spaces, will be treated in a forthcoming paper.

\section{Notation and Preparatory Results}\label{Notation and Preparatory Results}

In the whole section, $\mathcal{A}$ denotes an additive category with kernels and cokernels. Thus, $\mathcal{A}$ admits pushouts and pullbacks.  For later use we note the following (well-known) lemma, see e.g.~Richman, Walker \cite[Theorem 5]{RichmanWalker1977}.

\begin{lemma}\label{kernels and cokernels}\begin{compactitem} \item[(a)] If $g\colon Y\rightarrow Z$, $t\colon T\rightarrow Z$ are morphisms in $\mathcal{A}$ and $(P,p_T, p_Y)$ is their pullback, then there is a morphism $j\colon\ker g\rightarrow P$ making the diagram
\begin{diagram}[height=1.8em,width=2em]
\ker g              & \rTo^{j}   & P         & \rTo^{p_T}          & T        &         \\
\dTo^{\id_{\ker g}} &            &\dTo^{p_Y} & \text{\small\rm PB} & \dTo_{t} &         \\
\ker g              & \rTo_{i_g} & Y         & \rTo_{g}            & Z        &\;\;\;\: \\
\end{diagram}
commutative and being a kernel of $p_T$.\vspace{3pt}
\item[(b)] If $f\colon X\rightarrow Y$, $t\colon X\rightarrow T$ are morphisms in $\mathcal{A}$ and $(S,s_T, s_Y)$ is their pushout, then there is a morphism $c\colon S\rightarrow \cok f$ making the diagram
\begin{diagram}[height=1.8em,width=2em]
\;\;\;\:& X        & \rTo^{f}             & Y          & \rTo^{c_f} & \cok f              \\
        & \dTo^{t} & \text{\small\rm  PO} & \dTo_{s_Y} &            & \dTo_{\id_{\cok f}} \\
        & T        & \rTo_{s_T}           & S          & \rTo_{c}   & \cok f              \\
\end{diagram}
commutative and being a cokernel of $s_T$.
\end{compactitem}
\end{lemma}

In what follows we adopt the notation of Richman, Walker and call a morphism $f$ a \textit{kernel} if there is a morphism $g$ such that $f$ is a kernel of $g$. \textit{Cokernels} are defined dually. It is easy to see that $f$ is a kernel if and only if it is a kernel of its cokernel and dually $f$ is a cokernel if and only if it is the cokernel of its kernel. In the notation of Schneiders \cite[Definition 1.1.1]{Schneiders1999} a morphism $f\colon X\rightarrow Y$ in $\mathcal{A}$ is called \textit{strict} if the induced morphism $\bar{f}\colon\cok\ker f\rightarrow\ker\cok f$ is an isomorphism. From his remarks in \cite[Remark 1.1.2]{Schneiders1999} it follows that a morphisms $f$ is a strict epimorphism if and only if it is a cokernel and that $f$ is a strict monomorphism if and only if it is a kernel. Note that strict morphisms are called \textit{homomorphisms} by Wengenroth \cite{Wengenroth} in analogy to the notation of K\"othe \cite[p.~91]{KoetheI} (see also \cite[p.~307]{MeiseVogtEnglisch}) for the 
category of locally convex spaces. Kelly \cite[p.~124]{Kelly1969} introduced the notion of \textit{regular epimorphisms} within arbitrary categories and notes \cite[p.~126]{Kelly1969} that in the case of an additive category with kernels and cokernels a morphism is a regular epimorphism if and only if it is cokernel. Clearly, the notations of Schneiders and Kelly are more general than those of Richman, Walker. However, the latter notation will turn out to be the most convenient one for our purpose in this article.

The following definition is essential for the main result of this article and was also used by Richman, Walker \cite[p.~522]{RichmanWalker1977}.

\begin{definition}\label{lifting}\begin{compactitem} \item[(a)] A cokernel $g\colon Y\rightarrow X$ in $\mathcal{A}$ is said to be \textit{semi-stable}, if for every pullback square 
\begin{diagram}[height=1.8em,width=2em]
P          & \rTo^{p_T}       & T        & \\
\dTo^{p_Y} & \text{\small PB} & \dTo_{t} & \\
Y          & \rTo_{g}         & X        & \\
\end{diagram}
the morphism $p_T$ is also a cokernel.\vspace{3pt}
\item[(b)]A kernel $f\colon X\rightarrow Y$ $\mathcal{A}$ is said to be \textit{semi-stable}, if for every pushout square
\begin{diagram}[height=1.8em,width=2em]
& X        & \rTo^{f}         & Y          \\
& \dTo^{t} & \text{\small PO} & \dTo_{s_Y} \\
& T        & \rTo_{s_T}       & S          \\
\end{diagram}
the morphism $s_T$ is also a kernel.
\end{compactitem}
\end{definition}

\begin{rem}\label{l-strict remark} Because pullbacks and pushouts are transitive, retractions are stable under pullbacks, coretractions are stable under pushouts and isomorphisms are stable under both, we obtain the following.\vspace{5pt}
\begin{compactitem}
\item[(a)] In the situation of \ref{lifting}, $p_T$ and $s_T$  are again semi-stable.\vspace{3pt}
\item[(b)] Retractions are semi-stable cokernels.\vspace{3pt}
\item[(c)] Coretractions are semi-stable kernels.\vspace{3pt}
\item[(d)] Isomorphisms are semi-stable cokernels and semi-stable kernels.
\end{compactitem}
\end{rem}

The notion of a semi-stable cokernel coincides with that of a \textit{totally regular epimorphism}, as defined by Kelly \cite[p.~124]{Kelly1969} in the case of an additive category with kernels and cokernels. Kelly remarks that this notion was defined by Grothendieck (under a different name), see \cite[p.~124]{Kelly1969}. For later use we recall the following stability properties due to Kelly.

\begin{proposition}\label{criterion}(Kelly \cite[Proposition 5.11 and 5.12]{Kelly1969}) Let $f\colon X\rightarrow Y$ and $g\colon Y\rightarrow Z$ be morphisms in $\mathcal{A}$. Put $h:=g\circ f\colon X\rightarrow Z$.
\begin{compactitem}\vspace{3pt}
\item[(a)] If $f$ and $g$ are semi-stable cokernels, then $h$ is a semi-stable cokernel.\vspace{3pt}
\item[(b)] If $f$ and $g$ are semi-stable kernels, then $h$ is a semi-stable kernel.\vspace{3pt}
\item[(c)] If $h$ is a semi-stable cokernel, then $g$ is a semi-stable cokernel.\vspace{3pt}
\item[(d)] If $h$ is a semi-stable kernel, then $f$ is a semi-stable kernel.\vspace{3pt}
\end{compactitem}
\end{proposition}

To end this section let us remark that \ref{criterion} can be proved in an elementary way by very slight modifications of the proofs of Schneiders \cite[Propositions 1.1.7 and 1.1.8]{Schneiders1999} and by using the fact that in a diagram of the form
\begin{diagram}[height=1.8em,width=2em]
X          & \rTo^{f}  & Y          & \rTo^{g}  & Z          \\
\dTo^{a} &           & \dTo^{b} & \text{\small PB}          & \dTo_{c} \\
X'         & \rTo_{f'} & Y'         & \rTo_{g'} & Z'         \\
\end{diagram}
the left square is a pullback if and only if this is true for the exterior rectangle (cf.~Kelly \cite[Lemma 5.1]{Kelly1969}).

\section{Main Result}

In this section we show that every additive category $\mathcal{A}$ with kernels and cokernels admits a largest exact structure $\mathcal{E}$. The proof is constructive; to define $\mathcal{E}$ we use the terminology explained in \ref{lifting}. Then we check that $(\mathcal{A},\mathcal{E})$ is indeed an exact category. For the sake of completeness let us recall the definition of the latter, where we stick to the notation of B\"uhler, cf~also \ref{conflation-rem}.

\begin{definition}\label{exact category}(B\"uhler \cite[Definition 2.1]{Buehler2009}) Let $\mathcal{A}$ be an additive category. A \textit{kernel-cokernel pair} $(f,g)$ in $\mathcal{A}$ is a pair of composable morphisms $f\colon X\rightarrow Y$, $g\colon Y\rightarrow Z$ such that $f$ is a kernel of $g$ and $g$ is a cokernel of $f$. If a class $\mathcal{E}$ of kernel-cokernel pairs on $\mathcal{A}$ is fixed, an \textit{admissible monomorphism} is a morphism $f$ such that there exists a morphism $g$ with $(f,g)\in\mathcal{E}$. \textit{Admissible epimorphisms} are defined dually. An \textit{exact structure} on $\mathcal{A}$ is a class $\mathcal{E}$ of kernel-cokernel pairs which is closed under isomorphisms and satisfies the following axioms.
\medskip
\\\begin{compactitem}
\item[{[E0]}] For each object $X$, $\id_X\colon X\rightarrow X$ is an admissible monomorphism.\vspace{3pt}
\item[{[E0$^{\text{op}}$]}] For each object $X$, $\id_X\colon X\rightarrow X$ is an admissible epimorphism.\vspace{3pt}
\item[{[E1]}] If $g\colon Y\rightarrow Z$ and $g'\colon Z\rightarrow V$ are admissible monomorphisms, then $g'\circ g$ is an admissible monomorphism.\vspace{3pt}
\item[{[E1$^{\text{op}}$]}] If $g\colon Y\rightarrow Z$ and $g'\colon Z\rightarrow V$ are admissible epimorphisms, then $g'\circ g$ is a admissible epimorphism.\vspace{3pt}
\item[{[E2]}] If $f\colon X\rightarrow Y$ is an admissible monomorphism and $t\colon X\rightarrow T$ is a morphism, then the pushout
\begin{diagram}[height=1.8em,width=2em]
\,& X        & \rTo^{f}         & Y          \\
  & \dTo^{t} & \text{\small PO} & \dTo_{s_Y} \\
  & T        & \rTo_{s_T}       & S          \\
\end{diagram}
of $f$ and $t$ exists and $s_T$ is an admissible monomorphism.\vspace{3pt}
\item[{[E2$^{\text{op}}$]}] If $g\colon Y\rightarrow Z$ is an admissible epimorphism and $t\colon T\rightarrow Z$ is a morphism, then the pullback
\begin{diagram}[height=1.8em,width=2em]
P          & \rTo^{p_T}       & T        \\
\dTo^{p_Y} & \text{\small PB} & \dTo_{t} \\
Y          & \rTo_{g}         & Z        \\
\end{diagram}
of $g$ and $t$ exists and $p_T$ is an admissible epimorphism.
\end{compactitem}

\smallskip

An \textit{exact category} is an additive category $\mathcal{A}$ together with an exact structure $\mathcal{E}$; the kernel-cokernel pairs in $\mathcal{E}$ are called \textit{short exact sequences}.
\end{definition}

\begin{rem}\label{conflation-rem}\begin{compactitem}
\item[(a)] In the notation used by Keller \cite{Keller1990}, admissible monomorphisms are called \textit{inflations}, admissible epimorphisms are called \textit{deflations} and short exact sequences are called \textit{conflations}, see B\"uhler \cite[2.5]{Buehler2009} for more (historical) information concerning this terminology.\vspace{3pt}
\item[(b)] Some authors use the term \textit{short exact sequence} for what we call a kernel-cokernel pair and then call the sequences in $\mathcal{E}$ the \textit{admissible short exact sequences} or conflations (see (a)).
\end{compactitem}

\end{rem}

In what follows the proof of [E1$^{\text{op}}$] was inspired by that of Keller \cite[Proposition after A.1]{Keller1990}.

\begin{theorem}\label{main theorem} If $\mathcal{A}$ is an additive category with kernels and cokernels then the class
\begin{eqnarray*}
\mathcal{E}\!\!\!&=\:\big\{\,(e,f)\:;\: (e,f) \text{ is a short exact sequence},\;e \text{ is a semi-}\\
            &\hspace{28pt}\text{stable kernel}\text{ and } f \text{ is a semi-stable cokernel}\,\big\}\hspace{-35pt}
\end{eqnarray*}
is an exact structure on $\mathcal{A}$. Moreover, $\mathcal{E}$ is maximal in the sense that all exact structures on $\mathcal{A}$ are contained within it. In the notation of Richman, Walker \cite[p.~524]{RichmanWalker1977} the pairs in $\mathcal{E}$ are called stable.
\end{theorem}
\begin{proof}We show that $\mathcal{E}$ is closed under isomorphisms. Let $(f,g)\in \mathcal{E}$ and let 
\begin{diagram}[height=1.8em,width=2em]
X          & \rTo^{f}  & Y          & \rTo^{g}  & Z          \\
\dTo^{i_X} &           & \dTo_{i_Y} &           & \dTo_{i_Z} \\
X'         & \rTo_{f'} & Y'         & \rTo_{g'} & Z'         \\
\end{diagram}
be a commutative square in $\mathcal{A}$ with isomorphisms $i_X$, $i_Y$ and $i_Z$. Then $(f',g')$ belongs to $\mathcal{E}$. In fact, every commutative square
\begin{diagram}[height=1.8em,width=2em]
E           & \rTo^{h}  & F           \\
\dTo^{\phi} &           & \dTo_{\psi} \\
E'          & \rTo_{h'} & F'          \\
\end{diagram}
in $\mathcal{A}$ with isomorphisms $\phi$ and $\psi$ is a pullback square as well as a pushout square, hence $f'$ is a semi-stable kernel and $g'$ a semi-stable cokernel by \ref{l-strict remark}.(a) and it is easy to see that $f'$ is the kernel of $g'$. This shows $(f',g')\in \mathcal{E}$.
\medskip
\\By B\"uhler \cite[Remark 2.4]{Buehler2009} (cf.~Keller \cite[App.~A]{Keller1990}) the axioms in the definition of exact category are somewhat redundant and it is in fact enough to check the axioms [E0], [E0$^{\text{op}}$], [E1$^{\text{op}}$], [E2] and [E2$^{\text{op}}$] in order to show that $\mathcal{E}$ is an exact structure.
\smallskip
\\{[E0]} and [E0$^{\text{op}}$] are satisfied by \ref{l-strict remark}.(d).
\smallskip
\\{[E2$^{\text{op}}$]} Let $(f,g)\in \mathcal{E}$, assume that
\begin{diagram}[height=1.8em,width=2em]
& P          & \rTo^{p_T}       & T        &  \\
& \dTo^{p_Y} & \text{\small PB} & \dTo_{t} &  \\
& Y          & \rTo_{g}         & Z        &\,\\
\end{diagram}
is a pullback square. According to \ref{kernels and cokernels}.(a) we get a kernel $k\colon X\rightarrow P$ such that
\begin{diagram}[height=1.8em,width=2em]
X              & \rTo^{k}   & P         & \rTo^{p_T}          & T        &         \\
\dTo^{\id_{X}} &            &\dTo^{p_Y} & \text{\small\rm PB} & \dTo_{t} &         \\
X             & \rTo_{f} & Y         & \rTo_{g}            & Z        &\;\;\;\: \\
\end{diagram}
is commutative. Since $g$ is a semi-stable cokernel the same is true for $p_T$ by \ref{l-strict remark}.(a). Thus, the first row in the above diagram is a kernel-cokernel pair. Moreover, $p_Y\circ k=f$ and thus by \ref{criterion}.(d) the morphism $k$ is a semi-stable kernel which shows $(k,p_T)\in \mathcal{E}$.
\medskip
\\{[E2]} Since $\mathcal{A}$ has biproducts and cokernels, the pushout of any two morphisms does exist. The pair $(f,g)$ is in $\mathcal{E}$ if and only if $(g^{\text{op}},f^{\text{op}})$ is in $\mathcal{E}^{\text{op}}$. Then [E2] follows from {[E2$^{\text{op}}$]} by duality.
\medskip
\\{[E1$^{\text{op}}$]} Let $(f,g),(f',g')\in \mathcal{E}$ be pairs such that $g'\circ g$ is defined and let $k\colon K\rightarrow Y$ be a kernel of $g'\circ g$. Then $g'\circ g$ is a semi-stable cokernel by \ref{criterion}.(a) and $(k, g'\circ g)$ is a  short exact sequence. Thus it remains to be shown that $k$ is a semi-stable kernel.
\smallskip
\\Since $g'\circ g\circ k=0$, there exists a unique $\alpha\colon K\rightarrow X'$ with $f'\circ\alpha=g\circ k$.
\medskip
\\Claim A. The diagram
\begin{diagram}[height=1.8em,width=2em]
K        & \rTo^{\alpha}     & X'         \\
\dTo^{k} & \text{\small (1)} & \dTo_{f'} \\
Y        & \rTo_{g}          & Z         \\
\end{diagram}
is a pullback square.
\medskip
\\Let $l_Y\colon L\rightarrow Y$ and $l_{X'}\colon L\rightarrow X'$ be morphisms with $f'\circ l_{X'}=g\circ l_Y$. Then $g'\circ g\circ l_Y=g'\circ f'\circ l_{X'}=0$, hence there exists a unique $\eta\colon L\rightarrow K$ with $l_Y=k\circ\eta$. This yields $f'\circ l_{X'}=g\circ l_Y=g\circ k\circ\eta=f'\circ\alpha\circ\eta$ and from this it follows $l_{X'}=\alpha\circ\eta$, since $f'$ is a monomorphism. The morphism $\eta$ is unique with this property, since $k$ is a monomorphism, hence (1) is a pullback square. Thus, Claim A is established.
\medskip
\\Claim B. Let $(f,g)\in \mathcal{E}$ and
\begin{diagram}[height=1.8em,width=2em]
P          & \rTo^{p_R}       & R        \\
\dTo^{p_Y} & \text{\small PB} & \dTo_{r} \\
Y          & \rTo_{g}         & Z        \\
\end{diagram}
be a pullback square. Then $(\columnvec{-p_R}{\,\,\,\,\,p_Y},\rowvec{r}{g})\in \mathcal{E}$.
\medskip
\\By \ref{kernels and cokernels}.(a) we have a commutative diagram 
\begin{diagram}[height=1.8em,width=2em]
X            & \rTo^{k} & P          & \rTo^{p_R} & R      &      \\
\dTo^{\id_X} &          & \dTo_{p_Y} &            & \dTo_r &      \\
X            & \rTo_{f} & Y          & \rTo_g     & Z      &\;\!\!\\
\end{diagram}
such that $k$ is a kernel of $p_R$ and by [E2$^{\text{op}}$] the pair $(k,p_R)$ is in $\mathcal{E}$. We show that
\begin{diagram}[height=1.8em,width=2em]
X          & \rTo^{f}       & Y        \\
\dTo^{p_Y} &\text{\small(2)\!\!\!\!}  & \dTo_{\omega_Y} \\
P          & \rTo_{\scriptscriptstyle\!\!\!\!\!\columnvec{-p_R}{\,\,\,\,\,p_Y}}         & R\oplus Y        \\
\end{diagram}
is a pushout, where $\omega_Y$ denotes the canonical morphism. Let $l_Y\colon Y\rightarrow L$ and $l_P\colon P\rightarrow L$ be morphisms such that $l_Y\circ f=l_P\circ k$. Then $(l_Y\circ p_Y-l_P)\circ k=0$ holds and there is a unique morphism $\gamma\colon R\rightarrow L$ with $\gamma\circ p_R=l_Y\circ p_Y-l_P$ since $p_R$ is the cokernel of $k$. This in turn gives rise to a unique morphism $\mu\colon R\oplus Y\rightarrow S$ with $\gamma=\mu\circ\omega_R$ and $l_Y=\mu\circ\omega_Y$, where $\omega_R$ denotes the canonical morphism. We compute $l_P=l_Y\circ p_Y-\gamma\circ p_R=l_Y\circ p_Y-\mu\circ\omega_R\circ p_R=\mu\circ(\omega_Y\circ p_Y-\omega_R\circ p_R)=\mu\circ\columnvec{-p_R}{\,\,\,\,\,p_Y}$. The uniqueness of $\mu$ follows from the fact that $\gamma$ is unique and that $p_R$ is an epimorphism. Hence, (2) is a pushout and therefore $\columnvec{-p_R}{\,\,\,\,\,p_Y}$ is a semi-stable kernel. It remains to show that $[r\; g]$ is a cokernel of $\columnvec{-p_R}{\,\,\,\,\,p_Y}$ since then the claim 
follows by [E2]. We show that the pullback diagram in Claim B is also a pushout. Let $l_R\colon R\rightarrow L$ and $l_Y\colon Y\rightarrow L$ be morphisms with $l_Y\circ p_R=l_Y\circ p_Y$. Then we have $l_Y\circ f=l_Y\circ p_Y\circ k=l_R\circ p_R\circ k=0$, hence the universal property of the cokernel $g$ gives rise to a unique morphism $\lambda\colon Z\rightarrow L$ with $l_Y=\lambda\circ g$. In addition we have $l_R\circ p_R=l_Y\circ p_Y=\lambda\circ g\circ p_Y=\lambda\circ r\circ p_R$ and therefore $l_R=\lambda\circ r$ since $p_R$ is an epimorphism. This establishes Claim B.
\smallskip
\\By Claim B we know that the pair $(p,q)$ of morphisms $p:=\columnvec{-\alpha}{\,\,\,\,k}\colon K\rightarrow X'\oplus Y$ and $q:=\rowvec{f'}{g}\colon X'\oplus Y\rightarrow Z$ belongs to $\mathcal{E}$. We put $r:=\genfrac[]{0pt}{1}{f'\;\; 0\;\,}{\;\,0 \;\;\id_Y}$ and obtain the commutative diagram
\begin{diagram}[height=1.7em,width=3em]
 X'& \rTo^{f'}          & Z&          \\
\dTo^{\omega_{X'}}       & \text{\small (3)} &\dTo_{\omega_Z}          &          \\
 \!X'\!\oplus Y                      & \rTo_{r}         & Z\oplus Y                       &\;\;\;\;\;\\
\end{diagram}
where $\omega_{X'}$ and $\omega_Z$ are the canonical morphisms.
\medskip
\\Claim C. (3) is a pushout square.
\medskip
\\Let $l_{X'\oplus Y}\colon X'\oplus Y\rightarrow L$ and $l_Z\colon Z\rightarrow L$ be morphisms with $l_Z\circ f'=l_{X'\oplus Y}\circ{\omega_{X'}}$. Denote $l_Y:=l_{X'\oplus Y}\circ\omega_Y$ where $\omega_Y\colon Y\rightarrow Z\oplus Y $ is the canonical morphism. We have $\delta:=\rowvec{l_Z}{l_Y}\colon Z\oplus Y\rightarrow L$ and thus
\begin{eqnarray*}
l_{X'\oplus Y}\circ\omega_{X'}&=&l_Z\circ f'=\delta\circ\omega_Z\circ f'=\delta\circ r\circ\omega_{X'},\\
l_{X'\oplus Y}\circ\omega_Y&=&\delta\circ\omega_Y=\delta\circ\omega_Y\circ\pi_Y\circ r\circ\omega_Y=\delta\circ (id_{Z\oplus Y}-\omega_Z\circ\pi_Z)\circ r\circ\omega_Y=\delta\circ r\circ\omega_Y
\end{eqnarray*}
where $\pi_Y\colon Z\oplus Y\rightarrow Y$ and $\pi_Z\colon Z\oplus Y\rightarrow Z$ are the canonical morphisms.
\smallskip
\\Hence the universal property of the coproduct yields $l_{X'\oplus Y}=\delta\circ r$. The uniqueness of $\delta$ follows from the universal property of the coproduct, which yields Claim C.
\smallskip
\\Now $r$ is a semi-stable kernel and by \ref{criterion}.(c) the composition $r\circ p$ is also a semi-stable kernel. We put $\sigma:=\columnvec{-g}{\id_Y}$ and obtain
\begin{eqnarray*}
\;\;\;\;\;\;\;\;\;\;\;\;\;\;\;\;\;\;\;\;\;\;\;\;r\circ p=\genfrac[]{0pt}{1}{f'\;\,0\,}{\;\,0\;\id_Y}\genfrac[]{0pt}{1}{-\alpha}{k}=\genfrac[]{0pt}{1}{-f'\circ\alpha}{k}=\genfrac[]{0pt}{1}{-g\circ k}{k}=\sigma\circ k
\end{eqnarray*}
Since $r\circ p$ is a semi-stable kernel, it follows from \ref{criterion}.(d), that $k$ is a semi-stable kernel, which yields that $(k, g'\circ g)\in\mathcal{E}$ and thus that $\mathcal{E}$ is an exact structure on $\mathcal{A}$.
\medskip
\\It remains to check the maximality of $\mathcal{E}$. Let $\mathcal{E}'$ be a second exact structure on $\mathcal{A}$ and let $(f,g)\in\mathcal{E}'$. If
\begin{diagram}[height=1.8em,width=2em]
P          & \rTo^{p_T} & T        \\
\dTo^{p_Y} & \text{\small PB}  & \dTo_{t} \\
Y          & \rTo_{g}   & Z        \\
\end{diagram}
is a pullback square, the morphism $p_T$ is an admissible epimorphism by [E2$^{\text{op}}$], hence it is a cokernel (of its kernel), hence $g$ is a semi-stable cokernel. Analogously, by [E2] the morphism $f$ is a semi-stable kernel, which shows $(f,g)\in \mathcal{E}$. Hence we have $\mathcal{E}'\subseteq\mathcal{E}$.
\end{proof}

\section{Applications and Examples}\label{Applications}

In this last section we present two examples of additive categories with kernels and cokernels arising in functional analysis. We use standard locally convex notations and theory as presented in Meise and Vogt \cite{MeiseVogtEnglisch}, Jarchow \cite{Jarchow},  Bonet, P\'{e}rez Carreras \cite{BPC} and Floret, Wloka \cite{FloretWloka}. To simplify the notation, let us denote by LCS the category of locally convex topological vector spaces with linear and continuous maps as morphisms and by HD-LCS the full subcategory of spaces whose topology is Hausdorff.

\begin{example}\label{BOR} Let BOR be the full subcategory of LCS consisting of bornological spaces (cf.~\cite[\textsection{}\,23, 1.5 and \textsection{}\,11, 2.]{FloretWloka}). BOR is additive and since quotients of bornological spaces are again bornological, it has cokernels (cf.~\cite[\textsection{}\,23, 2.9]{FloretWloka}). Let $f\colon E\rightarrow F$ be a linear and continuous map in BOR. We consider the linear space $f^{-1}(0)$. If we consider $f$ as a morphism in LCS, $f^{-1}(0)$ endowed with the topology induced by $E$ would be a kernel of $f$. Unfortunately, this space is in general not bornological. However, $f^{-1}(0)$ endowed with the associated bornological topology w.r.t.~the induced one, (cf.~\cite[\textsection{}\,11, 2.2]{FloretWloka}) is bornological, we will denote this space by $f^{-1}(0)^{\text{BOR}}$. It is easy to check that this space together with the inclusion mapping is a kernel of $f$ in BOR.
\smallskip
\\From the above it follows that for an arbitrary morphism $f\colon E\rightarrow F$ in BOR, the cokernel of the kernel of $f$ is $E/f^{-1}(0)$ endowed with the quotient topology and that the kernel of the cokernel of $f$  is $f(E)^{\text{BOR}}$ w.r.t.~the topology induced by $F$. Thus, the canonical morphism $\bar{f}\colon \cok\ker f\rightarrow \ker\cok f$ is bijective and it is easy to see that it thus is both a monomorphism and epimorphism. Thus, BOR is semi-abelian in the notation of \cite{Rump2008}.
\smallskip
\\However, BOR is not quasi-abelian in the sense of \cite{Rump2008} and \cite{Schneiders1999}; Bonet, Dierolf \cite{BonetDierolf2005} constructed morphisms $f\colon E\rightarrow F$ and $g\colon G\rightarrow F$ such that $g$ is a cokernel but in the pullback square
\begin{diagram}[height=1.8em,width=2em]
P    & \rTo^{p_E}   & E\\
\dTo^{p_G} &  \text{\small PB}      & \dTo_f\\
G    & \rTo_g & F \\
\end{diagram}
$p_E$ fails to be a cokernel, that is $g$ is not semi-stable. Hence, the article of Bonet, Dierolf \cite{BonetDierolf2005} provides indeed a counter example to Raikov's conjecture (cf.~\cite{Rump2008}).
\end{example}

Theorem \ref{main theorem} provides a natural exact structure on BOR and hence enables us to use homological methods within this category, although BOR is neither quasi-abelian nor extension closed in LCS (see Roelcke, Dierolf \cite[Example 2.15]{RoelckeDierolf1981}). The same is true for the following example which is in a functional analytic sense more natural but unfortunately has even weaker properties in the sense of category theory.

\begin{example}\label{HD-BOR} By HD-BOR we denote the full subcategory of HD-LCS consisting of bornological spaces. Clearly, HD-BOR is a full subcategory of BOR. Since the preimage of zero under a continuous map between separated spaces is closed and since the associated bornological topology is finer than the starting topology (see \cite[\textsection{}\,11, 2.2]{FloretWloka}), the kernels in HD-BOR are those of BOR. Concerning cokernels this is not true. Let $f\colon E\rightarrow F$ be a morphism in HD-BOR. Then the cokernel of $f$ is the space $F/\overline{f(E)}$ endowed with the quotient topology. Hence, the cokernel of the kernel of $f$ is the space $E/f^{-1}(0)$ endowed with the quotient topology and the kernel of the cokernel of $f$ is $\overline{f(E)}^{\text{BOR}}$ w.r.t.~the topology induced by $F$.\vspace{-10pt}
\smallskip
\\By an example due to Grothendieck \cite{Grothendieck1954} (for details see Bonet, P\'{e}rez Carreras \cite[8.6.12]{BPC}), there exists a strict LB-space $(F,t)=\ind{n}(F_n,t_n)$ and a closed subspace $H\subseteq F$ such that there exists $u\in (H,s)'\backslash (H,t|_{H})'$ where $(H,s):=\ind{n}(H\cap F_n,t_n|_{H\cap F_n})$. Based on this example one may construct a mapping $f$ between spaces in HD-BOR, such that $\bar{f}$ is not an epimorphism, that is $\bar{f}(E)=f(E)$ is not dense in $\overline{f(E)}^{\text{BOR}}$. Hence the category HD-BOR is not even semi-abelian in Rump's notation.
\end{example}

\vspace{2pt}

\begin{acknowledgement}
The authors thank L.~Frerick and S.~Dierolf for many fruitful discussions during a seminar at the University of Trier organized by the first named author. Moreover, they thank J.~Bonet who explained the example used in \ref{HD-BOR} to the last named author and T.~B\"uhler for many helpful comments. Finally, they thank the referees for many detailed hints which provided simplifications of the proofs and for making them aware of the notation and the results in the articles \cite{Kelly1969} and \cite{RichmanWalker1977}.
\end{acknowledgement}


\vspace{2pt}


\begin{thebibliography}{[10]}

\providecommand{\WileyBibTextsc}{}
\let\textsc\WileyBibTextsc
\providecommand{\othercit}{}
\providecommand{\jr}[1]{#1}
\providecommand{\etal}{~et~al.}


\bibitem{BonetDierolf2005}
\textsc{J.~Bonet} and  \textsc{S.~Dierolf},
The pullback for bornological and ultrabornological spaces,
\jr{Note Mat.} \textbf{25}(1), 63--67 (2005/06).

\bibitem{Buehler2009}
\textsc{T.~B\"uhler},
Exact categories,
\jr{Expositiones Mathematicae} (2009), In Press, Accepted Manuscript, DOI: 10.1016/j.ex math.2009.04.004, URL: http://dx.doi.org/10.1016/j.exmath.2009.04.004

\bibitem{FloretWloka}
\textsc{K.~Floret} and  \textsc{J.~Wloka},
Einf\"uhrung in die {T}heorie der lokalkonvexen {R}\"aume,
Lecture Notes inMathematics, No. 56 (Springer-Verlag, Berlin, 1968).

\bibitem{Grothendieck1954}
\textsc{A.~Grothendieck},
Sur les espaces ({$F$}) et ({$DF$}),
\jr{Summa Brasil. Math.} \textbf{3}, 57--123 (1954).

\bibitem{Jarchow}
\textsc{H.~Jarchow},
Locally {C}onvex {S}paces (B. G. Teubner, Stuttgart, 1981).

\bibitem{Keller1990}
\textsc{B.~Keller},
Chain complexes and stable categories,
\jr{Manuscripta Math.} \textbf{67}(4), 379--417 (1990).

\bibitem{Kelly1969}
\textsc{G.~M.~Kelly},
Monomorphisms, epimorphisms, and pull-backs,
\jr{J. Austral. Math. Soc.} \textbf{9}, 124--142 (1969).

\bibitem{KoetheI}
\textsc{G.~K{\"o}the},
Topological vector spaces. {I}, (Springer-Verlag New York Inc., New York, 1969),
Translated from the German by D.~J.~H.~Garling. Die Grundlehren der mathematischen Wissenschaften, Band 159.

\bibitem{MeiseVogtEnglisch}
\textsc{R.~Meise} and  \textsc{D.~Vogt},
Introduction to functional analysis, Oxford Graduate Texts in Mathematics,
Vol.\,2 (The Clarendon Press Oxford University Press, New York, 1997), Translated from the German by M.~S.~Ramanujan and revised by the authors.

\bibitem{Palamodov1968}
\textsc{V.~P.~Palamodov},
The projective limit functor in the category of topological linear spaces,
\jr{Mat. Sb. \textbf{75} (1968) 567--603 (in Russian), English transl., Math. USSR Sbornik} \textbf{17}, 189--315 (1972).

\bibitem{Palamodov1971}
\textsc{V.~P.~Palamodov},
Homological methods in the theory of locally convex spaces,
\jr{Uspekhi Mat. Nauk \textbf{26} (1) (1971) 3--66 (in Russian), English transl., Russian Math. Surveys \textbf{26} (1)} pp.\,1--64 (1971).

\bibitem{BPC}
\textsc{P.~{P\'{e}rez Carreras}} and  \textsc{J.~Bonet},
Barrelled {L}ocally {C}onvex {S}paces (North--Holland Mathematics Studies \textbf{113}, 1987).

\bibitem{Prosmans2000}
\textsc{F.~Prosmans},
Derived categories for functional analysis,
\jr{Publ. Res. Inst. Math. Sci.} \textbf{36}(5-6), 19--83 (2000).

\bibitem{Quillen1973}
\textsc{D.~Quillen},
Higher algebraic {$K$}-theory. {I},
in: Algebraic {$K$}-theory, {I}: {H}igher {$K$}-theories ({P}roc. {C}onf., {B}attelle {M}emorial {I}nst., {S}eattle, {W}ash., 1972),  (Springer, Berlin, 1973),  pp.\,85--147. Lecture Notes in Math., Vol. 341.

\bibitem{RichmanWalker1977}
\textsc{F.~Richman} and  \textsc{E.~A.~Walker},
Ext in pre-Abelian categories,
\jr{Pacific J. Math.} \textbf{71}(2), 521--535 (1977).

\bibitem{RoelckeDierolf1981}
\textsc{W.~Roelcke} and  \textsc{S.~Dierolf},
On the three-space-problem for topological vector spaces,
\jr{Collect. Math.} \textbf{32}(1), 13--35 (1981).

\bibitem{Rump2008}
\textsc{W.~Rump},
A counterexample to {R}aikov's conjecture,
\jr{Bull. Lond. Math. Soc.} \textbf{40}(6), 985--994 (2008).

\bibitem{Schneiders1999}
\textsc{J.~P.~Schneiders},
Quasi-abelian categories and sheaves,
\jr{M\'em. Soc. Math. Fr. (N.S.)}(76) (1999).

\bibitem{VogtLectures}
\textsc{D.~Vogt},
Lectures on projective spectra of ({DF})-spaces,
Seminar lectures, {W}uppertal, 1987.

\bibitem{Wengenroth}
\textsc{J.~Wengenroth},
Derived {F}unctors in {F}unctional {A}nalysis (Lecture {N}otes in {M}athematics \textbf{1810}, Springer, Berlin, 2003).

\end{thebibliography}
\end{document}